\documentclass[11pt]{article}
\usepackage{latexsym, amsmath, amssymb}

\usepackage{amsmath}
\usepackage{amsfonts}
\usepackage{amssymb}
\usepackage{amscd}
\usepackage{amsthm}
\usepackage{fancyhdr}
\theoremstyle{plain}
\newtheorem{theorem}{Theorem}[section]
\newtheorem{proposition}[theorem]{Proposition}
\newtheorem{remark}[theorem]{Remark}

\newtheorem{definition}[theorem]{Definition}
\newtheorem{corollary}[theorem]{Corollary}

\def\Q{\mathbb Q}

\newcommand\sO{{\mathcal O}}

\newcommand\sS{{\mathcal S}}

  \def \tab#1{\kern #1 truein}

  \def\E{\hbox{${\cal E}$}}
  
  \def\G{\hbox{${\cal G}$}}

  \def\H{\hbox{${\cal H}$}}
  
  \def\Q{\hbox{${\cal Q}$}}

\begin{document}
 \title{Qregularity and an Extension of Evans-Griffiths Criterion to Vector Bundles on Quadrics}
\author{Edoardo Ballico and Francesco Malaspina
\vspace{6pt}\\
{\small   Universit\`a di Trento}\\
{\small\it 38050 Povo (TN), Italy}\\
{\small\it e-mail: ballico@science.unitn.it}\\
\vspace{6pt}\\
{\small  Politecnico di Torino}\\
{\small\it  Corso Duca degli Abruzzi 24, 10129 Torino, Italy}\\
{\small\it e-mail: alagalois@yahoo.it}}
    \maketitle \def\thefootnote{}
\footnote{Mathematics Subject Classification 2000: 14F05, 14J60. \\
keywords: Spinor bundles; coherent sheaves on quadric hypersurfaces; Castelnuovo-Mumford regularity.}
  \begin{abstract} Here we define the concept of Qregularity for  coherent sheaves on quadrics. In this setting we prove analogs of some classical properties. We compare the Qregularity of coherent sheaves on $\Q_n\subset \mathbb P^{n+1}$ with the Castelnuovo-Mumford regularity of their extension by zero in $\mathbb P^{n+1}$. We also classify the coherent sheaves with Qregularity $-\infty$. We use our notion of Qregularity in order to prove an extension of Evans-Griffiths criterion to vector bundles on Quadrics. In particular we get a new and simple proof of the Kn\"{o}rrer's characterization of ACM bundles.
  \end{abstract}
  \section{Introduction}
  In chapter $14$ of \cite{m} Mumford introduced the concept of regularity for a coherent sheaf on a projective space $\mathbb P^{n}$. Since then, Castelnuovo-Mumford regularity has become a fundamental invariant and was investigate by several people. Chipalkatti in \cite{c} has generalized this notion on grassmannians and Hoffman and Wang in \cite{hw}
  on multiprojective spaces. Costa and Mir\'o-Roig in \cite{cm1} and \cite{cm2} give a definition of regularity for coherent sheaves on $n$-dimensional smooth projective varieties with an $n$-block collection.\\

  The aim of this note is to introduce  a very simple and natural concept of regularity (the Qregularity) on a quadric hypersurface.\\
  If we consider the following geometric collection on $\mathbb P^{n}$:
  $$(\E_0, \dots, \E_n)= (\sO_{\mathbb{P}^n}(-n), \sO_{\mathbb{P}^n}(-n+1),\dots , \sO_{\mathbb{P}^n}),$$
    we have that a coherent sheaf F on $\mathbb{P}^n$ is said to be $m$-regular
according to Castelnuovo-Mumford
if for any $i=1,\dots ,n$
  $$H^i(F(m)\otimes\E_{n-i})=0.$$
  On $\Q_n$ we consider the following $n$-block collection:
   $$(\G_0, \dots, \G_n)= (\G_0, \sO_{\Q_n}(-n+1),\dots , \sO_{\Q_n}),$$
   where $\G_0=(\Sigma(-n))$ if $n$ is odd or where $\G_0=(\Sigma_1(-n), \Sigma_2(-n))$ if $n$ is even and $\Sigma_*$ are the spinor bundles (for generalities about spinor bundles see \cite{o1}).\\
   We say that a coherent sheaf $F$ is $m$-Qregular  if for any $i=1,\dots ,n$
  $$H^i(F(m)\otimes\G_{n-i})=0.$$
  The interesting fact is that on $\Q_2\cong \mathbb P^1\times \mathbb P^1$ our definition of $m$-Qregularity coincides with the definition of  $(m,m)$-regularity on $\mathbb P^1\times \mathbb P^1$ by Hoffman and Wang (see \cite{hw}).
  So we use the results on $\mathbb P^1\times \mathbb P^1$ as the starting step in order to prove on $\Q_n$,
  by induction on $n$, analogs of the classical properties on $\mathbb P^{n+1}$.\\
  Next we give some equivalent condition of Qregularity. We compare the Qregularity of coherent sheaves on $\Q_n\subset \mathbb P^{n+1}$ with the Castelnuovo-Mumford regularity of their extension by zero in $\mathbb P^{n+1}$. We also classify the coherent sheaves with Qregularity $-\infty$ as those with finite support.

  The second aim to this paper is to apply our notion of Qregularity in order to investigate under what circumstances a vector bundle can be decomposed into a direct sum of line bundles.\\
 A well known result by Horrocks (see \cite{Ho}) characterizes the vector bundles without intermediate cohomology on a projective space as direct sum of line bundles.
This criterion fails on  more general varieties. In fact there exist non-split vector bundles  without intermediate cohomology. This bundles are called ACM bundles.\\
On $\mathbb P^n$, Evans and Griffith (see \cite{e}) have improved Horrocks' criterion:
\begin{theorem}[Evans-Griffith]
 Let $E$ be a vector bundle of rank $r$ on $\mathbb P^n$, $n\geq 2$, then $\E$ splits if and only if $\forall i=1, ..., r-1$ $\forall k\in \mathbb Z$,
 $$h^i(E(k))=0.$$
 \end{theorem}
On a quadric hypersurface $\Q_n$  there is a theorem by Kn\"{o}rrer that classifies all the ACM bundles (see \cite{Kn}) as  direct sums of line bundles and spinor bundles (up to a twist).\\
Ottaviani has generalized Horrocks criterion to  quadrics and Grassmanniann giving cohomological splitting conditions for vector bundles (see \cite{o2}).\\
Our main result is to extend Evans-Griffiths criterion to vector bundles on quadrics. We improve  Kn\"{o}rrer's theorem in the following way:
\begin{theorem}\label{t1} Let $E$ be a rank $r$ vector bundle on $\Q_n$. Then the following conditions are equivalents:
\begin{enumerate}
\item $H^i_*(E) =0$ for any $i=1, \dots, r-1$ and $H^{n-1}_*(E) =0$ ,
\item $E$ is a direct sum of line bundles and spinor bundles with some twist.
\end{enumerate}
\end{theorem}
In particular we get a new and simple proof of the Kn\"{o}rrer's characterization of ACM bundles.\\
 The hypothesis $H^{n-1}_*(E) =0$ does not appear in the Evans-Griffiths criterion on $\mathbb P^{n}$. On $\Q_n$ it is necessary.\\
       In fact we can find many indecomposable bundles with  $H^{1}_*(E)=\dots =H^{n-2}_*(E)=0$ but
      $H^{n-1}_*(E) \not=0$ (see \cite{h} or \cite{ml}).\\
      Then we specialize to the case: rank $E=2$.\\
      We prove that if a Qregular rank $2$ bundle $E$ has
      $H^1(E(-2))=H^1(E(c_1))=0$, then it is a direct sum of line bundles and spinor bundles with some twist.
In particular if $n>4$, $E\cong\sO\oplus\sO(c_1)$.\\

      We work over an algebraically
closed field with characteristic
zero. We only need the characteristic zero assumption to prove Theorem \ref{t1}
and Proposition \ref{p4}
because in their proofs we will use Le Potier vanishing theorem.\\
We thanks E. Arrondo. He showed us the connection between the Qregularity and the splitting criteria for vector bundles.

  \section{$m$-Qregular coherent sheaves: definition and properties}

  Let us consider a smooth quadric hypersurface $\Q_n$ in $\mathbb P^{n+1}$.\\
We use the unified notation $\Sigma_*$ meaning that for even $n$ both the spinor bundles $\Sigma_1$ and $\Sigma_2$ are considered, and for $n$ odd, the spinor bundle $\Sigma$.\\
We follow the notation of \cite{cm1} so the spinor bundles are twisted by $1$ with respect to those of \cite{o1} ($\Sigma_*=\sS_*(1)$)
  \begin{definition} A  coherent sheaf $F$ on $\Q_n$ ($n\geq 2$) is said $m$-Qregular if $H^i(F(m-i))=0$ for $i=1\dots n-1$ and $H^n(F(m)\otimes \Sigma_*(-n))=0$.\\
  We will say Qregular in order to $0$-Qregular.
  \end{definition}
  \begin{remark}A  coherent sheaf $F$ on $\Q_2\cong \mathbb P^1\times \mathbb P^1$ is $m$-Qregular if $H^1(F(m-1,m-1))=0$  and $H^2(F(m,m)\otimes \Sigma_*(-2,-2))=0$.\\ Since  we have $\Sigma_1\cong\sO(1,0)$ and $\Sigma_2\cong\sO(0,1)$, our conditions become $H^1(F(m-1,m-1))=0$, $H^2(F(m-1,m-2))=0$ and $H^2(F(m-2,m-1))=0$.\\
  So the definition of $m$-Qregularity coincides with the definition of  $(m,m)$-regularity on $\mathbb P^1\times \mathbb P^1$ by Hoffman and Wang (see \cite{hw}).
  \end{remark}
  \begin{proposition}\label{p1} Let $F$ be an $m$-Qregular coherent sheaf on $\Q_n$ $(n\geq 2)$ then
  \begin{enumerate}
  \item $F$ is $k$-Qregular for $k\geq m$.\\
  \item $H^0(F(k))$ is spanned by $H^0(F(k-1))\otimes H^0(\sO(1))$ if $k>m$.\\
  \end{enumerate}

  \begin{proof}We use induction on $n$:\\
  for $n=2$, $(1)$ comes from \cite{hw} Proposition $2.7.$ and  $(2)$ from \cite{hw} Proposition $2.8.$\\
  Let us study the case $n=3$. Let $F$ be an $m$-Qregular coherent sheaf on $\Q_3$.
   We have the following exact sequence for any integer $k$:
  $$0\rightarrow F(k-1)\rightarrow F(k) \rightarrow F_{|\Q_2}(k,k)\rightarrow 0.$$
In particular we get
$$0=H^1(F(m-1)) \rightarrow H^1(F_{|\Q_2}(m-1,m-1))\rightarrow H^2(F(m-2))= 0.$$
This implies that $H^1(F_{|\Q_2}(m-1,m-1))=0$.\\
Let us consider now the exact sequence on $\Q_3$:
$$ 0\rightarrow F(m-2)\otimes \Sigma(-1)\rightarrow F(m-2)^4 \rightarrow F(m-2)\otimes\Sigma\rightarrow 0.$$
Since $H^2(F(m-2))=0$  and $H^3(F(m)\otimes \Sigma(-3))=0$, we also have $H^2(F(m)\otimes \Sigma(-2))=0$.\\
If we tensorize by $F(m)$ the exact sequence
$$0\rightarrow \Sigma(-3)\rightarrow \Sigma(-2) \rightarrow \Sigma_{|\Q_2}(-2,-2)\rightarrow 0$$
we get
$$0\rightarrow F(m)\otimes \Sigma(-3)\rightarrow F(m)\otimes\Sigma(-2) \rightarrow (F_{|\Q_2}(m)\otimes\Sigma_1(-2,-2))\oplus(F_{|\Q_2}(m)\otimes\Sigma_2(-2,-2))\rightarrow 0.$$
So we obtain
 $$H^2(F(m)\otimes\Sigma(-2)) \rightarrow H^2(F(m)\otimes\Sigma_1(-2,-2))\oplus H^2(F(m)\otimes\Sigma_2(-2,-2))\rightarrow H^3(F(m)\otimes\Sigma(-3))$$ and hence $ H^2(F_{|\Q_2}(m)\otimes\Sigma_1(-2,-2))= H^2(F_{|\Q_2}(m)\otimes\Sigma_2(-2,-2))=0$.\\
 We can conclude that if $F$ is $m$-Qregular on $\Q_3$, then $F_{|\Q_2}$ is an $m$-Qregular on $\Q_2$.\\
 We have hence $(1)$ and $(2)$ for $F_{|\Q_2}$.\\
 Moreover from the exact sequence on $\Q_2$:
 $$ 0\rightarrow F_{|\Q_2}(m-1,m-1)\otimes \Sigma_1(-1,-1)\rightarrow F_{|\Q_2}(m-1,m-1)^2 \rightarrow F_{|\Q_2}(m-1,m-1)\otimes\Sigma_2\rightarrow 0,$$
 we get
 $$ H^2(F_{|\Q_2}(m,m)\otimes \Sigma_1(-2,-2))\rightarrow H^2(F_{|\Q_2}(m-1,m-1))^2 \rightarrow H^2(F_{|\Q_2}(m+1,m+1)\otimes\Sigma_2(-2,-2)).$$
 By $(1)$ for $F_{|\Q_2}$, the last group is $0$; by $m$-Qregularity the first group is $0$. Therefore $H^2(F_{|\Q_2}(m-1,m-1))=0$.\\
 Let us consider now
  $$H^i(F(m-i))\rightarrow H^i(F(m+1-i)) \rightarrow H^i(F_{|\Q_2}(m+1-i,m+1-i)).$$
  By $(1)$ for $F_{|\Q_2}$ if $i=1$, or by the above argument if $i=2$, the last group is $0$; by $m$-Qregularity ,  the first is $0$ if $i=1,2 $. Therefore $H^1(F(m+1-1))=H^2(F(m+1-2))=0$.\\
  Moreover from the sequence
  $$H^3(F(m)\otimes \Sigma(-3))\rightarrow H^3(F(m+1)\otimes\Sigma(-3)) \rightarrow 0$$
  we see that also $H^3(F(m+1)\otimes\Sigma(-3))=0$.\\
  We can conclude that $F$ is $(m+1)$-Qregular. Continuing in this way we prove $(1)$ for $F$.\\
  To get $(2)$ we consider the following diagramm:
  $$\begin{array}{ccc}

H^0(F(k-1))\otimes H^0(\sO_{\Q_3}(1))& \xrightarrow{\sigma}& H^0(F_{|\Q_2}(k-1,k-1))\otimes H^0(\sO_{\Q_2}(1,1))\\
 \downarrow \scriptsize{\mu}& &\downarrow \scriptsize{\tau}\\
 H^0(F(k))&\xrightarrow{\nu}& H^0(F_{|\Q_2}(k,k))

 \end{array}$$
  Note that $\sigma$ is surjective if $k>m$ because $H^1(F(k-2))=0$.\\
  Let us prove that if $k>m$ then $\tau$ is surjective:\\
  From \cite{hw} Proposition $2.8.$ we know that $H^0(F_{|\Q_2}(k,k))$ is spanned by $$H^0(F_{|\Q_2}(k-1,k))\otimes H^0(\sO_{|\Q_2}(1,0))$$ and also by $$H^0(F_{|\Q_2}(k,k-1))\otimes H^0(\sO_{|\Q_2}(0,1))$$ so the maps
  $$H^0(F_{|\Q_2}(k-1,k-1))\otimes H^0(\sO_{|\Q_2}(1,0))\otimes H^0(\sO_{|\Q_2}(0,1))\rightarrow H^0(F_{|\Q_2}(k-1,k))\otimes H^0(\sO_{|\Q_2}(1,0))$$ and
 $$H^0(F_{|\Q_2}(k-1,k))\otimes H^0(\sO_{|\Q_2}(1,0)) \rightarrow H^0(F_{|\Q_2}(k,k))$$ are both surjective. Hence their  composition is surjective.\\
  Now since we have the surjection $$H^0(\sO_{|\Q_2}(1,0))\otimes H^0(\sO_{|\Q_2}(0,1))\rightarrow H^0(\sO_{|\Q_2}(1,1))$$ we have that also $$H^0(F_{|\Q_2}(k-1,k-1))\otimes H^0(\sO_{|\Q_2}(1,1))\xrightarrow{\tau} H^0(F_{|\Q_2}(k,k)),$$ is a surjection.\\
  Since both $\sigma$ and $\tau$ are surjective, we can see as in \cite{m} page $100$ that $\mu$ is also surjective.\\

  Let assume $(1)$ and $(2)$ on $\Q_{2n-1}$. We prove it on $\Q_{2n}$.\\
  Let $F$ be an $m$-Qregular coherent sheaf on $\Q_{2n}$.
   We have the following exact sequence for any integer $k$:
  $$0\rightarrow F(k-1)\rightarrow F(k) \rightarrow F_{|Q_{2n-1}}(k)\rightarrow 0.$$
In particular we get
$$H^i(F(m-i)) \rightarrow H^i(F_{|Q_{2n-1}}(m-i))\rightarrow H^{i+1}(F(m-i-1)).$$
This implies that $H^i(F_{|\Q_{2n-1}}(m-i))=0$ for $i=1, \dots , 2n-2$.\\
Let us consider now the exact sequences on $\Q_{2n}$:
$$ 0\rightarrow F(m-2n+1)\otimes \Sigma_1(-1)\rightarrow F(m-2n+1)^{2^n} \rightarrow F(m-2n+1)\otimes\Sigma_2\rightarrow 0$$
and
$$ 0\rightarrow F(m-2n+1)\otimes \Sigma_2(-1)\rightarrow F(m-2n+1)^{2^n} \rightarrow F(m-2n+1)\otimes\Sigma_1\rightarrow 0.$$
Since $H^{2n-1}(F(m-2n+1))=0$  and $H^{2n}(F(m)\otimes \Sigma_*(-2n))=0$, we also have\\ $H^{2n-1}(F(m)\otimes \Sigma_*(-2n+1))=0$.\\
If we tensorize by $F(m)$ the exact sequence
$$0\rightarrow \Sigma_1(-2n)\rightarrow \Sigma_1(-2n+1) \rightarrow {\Sigma_1}_{|\Q_{2n-1}}(-2n+1)\rightarrow 0$$
we get
$$0\rightarrow F(m)\otimes \Sigma_1(-2n)\rightarrow F(m)\otimes\Sigma_1(-2n+1) \rightarrow F_{|\\Q_{2n-1}}(m)\otimes\Sigma_1(-2n+1)\rightarrow 0.$$
So we obtain
 $$H^{2n-1}(F(m)\otimes\Sigma_1(-2n+1)) \rightarrow H^{2n-1}(F_{|\Q_{2n-1}}(m)\otimes\Sigma_1(-2n+1))\rightarrow H^{2n}(F(m)\otimes\Sigma_1(-2n))$$ and hence $ H^{2n-1}(F_{|\\Q_{2n-1}}(m)\otimes\Sigma_1(-2n+1))=0$.\\
 We can conclude that if $F$ is  an $m$-Qregular sheaf on $\Q_{2n}$, then $F_{|\Q_{2n-1}}$ is an $m$-Qregular on $\Q_{2n-1}$.\\
 We have hence by the induction hypothesis $(1)$ and $(2)$ for $F_{|\Q_{2n-1}}$.\\
 Moreover from the exact sequence on $\Q_{2n-1}$:
 $$ 0\rightarrow F_{|\Q_{2n-1}}(m)\otimes \Sigma(-2n+1)\rightarrow F_{|\Q_{2n-1}}(m-2n+2)^{2^n} \rightarrow F_{|\Q_{2n-1}}(m+1)\otimes\Sigma(-2n+1)\rightarrow 0,$$
 we get
 $$ H^{2n-1}(F_{|\Q_{2n-1}}(m)\otimes \Sigma(-2n+1))\rightarrow H^{2n-1}(F_{|\Q_{2n-1}}(m-2n+2))^{2^n} \rightarrow$$ $$\rightarrow H^{2n-1}(F_{|\Q_{2n-1}}(m+1)\otimes\Sigma(-2n+1)).$$
 By $(1)$ for $F_{|\Q_{2n-1}}$, the last group is $0$; by $m$-Qregularity the first group is $0$. Therefore $H^{2n-1}(F_{|\Q_{2n-1}}(m-2n+2))=0$.\\
 Let us consider now the exact sequence
  $$H^i(F(m-i))\rightarrow H^i(F(m+1-i)) \rightarrow H^i(F_{|\Q_{2n-1}}(m+1-i)).$$
  By $(1)$ for $F_{|\\Q_{2n-1}}$ if $i=1, \dots 2n-2$, or by the above argument if $i=2n-1$, the last group is $0$; by $m$-Qregularity ,  the first is $0$ if $i=1, \dots 2n-1 $. Therefore $H^i(F(m+1-i))=0$ for $i=1, \dots 2n-1$.\\
  Moreover from the exact sequence
  $$H^{2n}(F(m)\otimes \Sigma(-2n))\rightarrow H^{2n}(F(m+1)\otimes\Sigma(-2n)) \rightarrow 0$$
  we also see that  $H^{2n}(F(m+1)\otimes\Sigma(-2n))=0$.\\
  We can conclude that $F$ is $(m+1)$-Qregular. Continuing in this way we prove $(1)$ for $F$.\\

  To get $(2)$ we consider the following diagramm:
$$  \begin{array}{ccc}
 H^0(F(k-1))\otimes H^0(\sO_{\Q_{2n}}(1))& \xrightarrow{\sigma}& H^0(F_{|\Q_{2n-1}}(k-1))\otimes H^0(\sO_{\Q_{2n-1}}(1))\\
 \downarrow \scriptsize{\mu}& &\downarrow \scriptsize{\tau}\\
 H^0(F(k))&\xrightarrow{\nu}& H^0(F_{|\Q_{2n-1}}(k))

 \end{array}$$
  Note that $\sigma$ is surjective if $k>m$ because $H^1(F(k-2))=0$.\\
 Moreover if $k>m$ also $\tau$ is surjective by $(2)$ for $F_{|\Q_{2n-1}}$.\\
  Since both $\sigma$ and $\tau$ are surjective we can see as in \cite{m} page $100$ that $\mu$ is also surjective.\\
  In a very similar way we can prove $(1)$ and $(2)$ from $\Q_{2n}$ to $\Q_{2n+1}$.

  \end{proof}
  \end{proposition}

  We can give some equivalent definition of $m$-Qregular coherent sheaves:\\
  \begin{proposition}\label{d-e} Let $F$ be a coherent sheaf on $\Q_n$. 
  The following conditions are equivalent:\\
   \begin{enumerate}
   \item $F$ is $m$-Qregular.
   \item $H^i(F(m-i))=0$ for $i=1,\dots, n-1$, $H^{n-1}(F(m)\otimes \Sigma_*(-n+1))=0$ and\\
    $H^n(F(m-n+1))=0$.\\
   \end{enumerate}
   \begin{proof}
   We look at the exact sequences
   $$ 0\rightarrow F(k)\otimes \Sigma_*(-1)\rightarrow F(k)^{2^{([n/2]+1)}} \rightarrow F(k)\otimes\Sigma_*\rightarrow 0.$$
   $(1)\Rightarrow (2).$   Let $F$ be $m$-Qregular  then by (\ref{p1}) is also $(m+1)$-Qregular.
    We see that $$H^n(F(m+1)\otimes \Sigma_*(-n-1))=H^n(F(m+1)\otimes \Sigma_*(-n))=0\Rightarrow H^n(F(m+1-n))=0$$
    and
    $$H^n(F(m+1)\otimes \Sigma_*(-n-1))=H^{n-1}(F(m+1-n))=0\Rightarrow H^{n-1}(F(m+1)\otimes \Sigma_*(-n))=0.$$
    So we have $(2)$.\\
    $(2)\Rightarrow (1).$ Let $F$ be a coherent sheaf which satisfies ($2$).
     From
    $$H^{n-1}(F(m)\otimes \Sigma_*(-n+1))= H^{n}(F(m-n+1))=0\Rightarrow H^{n}(F(m)\otimes\Sigma_*(-n))=0,$$
    we see that $F$ is $m$-regular.
   \end{proof}
   \end{proposition}
   Now we show that the Qregular coherent sheaf are globally generated:
   \begin{proposition} Any Qregular coherent sheaf $F$ on $\Q_n$ is globally generated.\end{proposition}\begin{proof} We need to prove that the evaluation map  $$\varphi:H^0(F)\otimes\sO_{Q_n}\to F$$ is surjective. This is equivalent to prove that its tensor product with $id_\Sigma$ is surjective, because this would imply that$$\varphi\otimes id_\Sigma\otimes id_{\Sigma^\vee}:H^0(F)\otimes\Sigma\otimes\Sigma^\vee\to F\otimes\Sigma\otimes\Sigma^\vee$$ is surjective, and restricting to the component of endomorphims of $\Sigma$ of trace zero we get that$\varphi$ is surjective. We thus observe that we have a commutative diagram
    $$\begin{matrix}H^0(F)\otimes H^0(\Sigma)\otimes\sO_{Q_n}&\xrightarrow{\eta}&H^0(F\otimes\Sigma)\otimes\sO_{Q_n}\\
    \downarrow&&\downarrow\psi\\H^0(F)\otimes\Sigma&\xrightarrow{\varphi\otimes id_\Sigma}&F\otimes\Sigma\end{matrix}$$so that it is enough to prove that $\eta$ and $\psi$ are surjective. This follows from Proposition
    \ref{d-e}
     which implies that $H^1(F\otimes\Sigma_*(-1)=0$ (hence $\eta$ is surjective) and more generally $F\otimes\Sigma$ is regular, (and hence it is globally generated and $\psi$ is surjective).\end{proof}

  \section{Qregularity on $\Q_n$}
  \begin{definition}Let $F$ be a coherent sheaf on $\Q_n$. We define the Qregularity of $F$, $Qreg (F)$, as the least integer $m$ such that $F$ is $m$-Qregular. We set $Qreg (F)=-\infty$ if there is no such integer.
  \end{definition}
  \begin{remark}\label{r1} On $\Q_n$, we show that $Qreg (\sO)=Qreg(\Sigma_*)=0$.\\
  $\sO$ and $\Sigma_*$ are ACM bundles.\\
   $H^n(\sO\otimes\Sigma_*(-n))\cong H^0(\Sigma^{\vee})=H^0(\Sigma(-1))=0$ and \\
   $H^n(\Sigma_*\otimes\Sigma_*(-n))\cong H^0(\Sigma_*\otimes\Sigma_*(-2))=0$.\\
   So $\sO$ and $\Sigma_*$ are $0$-Qregular.\\
   $H^n(\sO(-1)\otimes\Sigma_*(-n))\cong H^0(\Sigma^{\vee}(1))=H^0(\Sigma)\not=0$ implies that $\sO$ is not $(-1)$-Qregular.\\
   $H^n(\Sigma_*(-1)\otimes\Sigma_*(-n))\cong H^0(\Sigma^{\vee}_*(1)\otimes\Sigma_*(-1))$.\\
   Now from the exact sequence
$$ 0\rightarrow \Sigma_*^\vee\otimes \Sigma_*(-1)\rightarrow \Sigma_*^\vee\otimes\sO^{2^n} \rightarrow \Sigma_*^\vee\otimes\Sigma_*\rightarrow 0,$$ since $H^0(\Sigma_*^\vee)=H^1(\Sigma_*^\vee)$ we see that
   $H^0(\Sigma^{\vee}_*(1)\otimes\Sigma_*(-1))\cong H^1(\Sigma_*(-1)\otimes\Sigma_*^{\vee})$.\\
   But $H^1(\Sigma_*(-1)\otimes \Sigma_*^{\vee})\not=0$ by \cite{cm3} Lemma $4.1.$ This means that $\Sigma_*$ is not $(-1)$-Qregular.\\
   \end{remark}
   \begin{remark}Let $$0\rightarrow F_1\rightarrow F_2\rightarrow F_3\rightarrow 0$$ be an exact sequence of coherent sheaves.\\
   Then $Qreg (F_2)\leq \max \{Qreg(F_1),Qreg(F_3)\}$.\\
   Let $F$ and $G$ be coherent sheaves.\\
   Then  $Qreg (F\oplus G)= \max \{Qreg(F),Qreg(G)\}$.\\
   \end{remark}
   Let $F$ be a coherent sheaf on $\Q_n$ ($n>1$) and let $i_* F$ be its extension by zero in the embedding $\Q_n\hookrightarrow \mathbb P^{n+1}$. We can compare the Qregularity of $F$ with the regularity in the sense of Castelnuovo-Mumford of $i_*F$. We recall the definition:\\
   \begin{definition} A  coherent sheaf $F$ on $\Q_n$  is said $m$-regular in the sense of Castelnuovo-Mumford if $H^i(\mathbb P^{n+1},i_* F(m-i))=0$ for $i=1\dots n+1$.\\
   $Reg (F)$ is the least integer $m$ such that $F$ is $m$-regular. We set $reg (F)=-\infty$ if there is no such integer.
   \end{definition}
   \begin{proposition}\label{p2}Let $\Q_n\hookrightarrow \mathbb P^{n+1}$ be a quadric hypersurface ($n>1$). For any coherent sheaf $F$ we have:\\
   $$Qreg(F)\leq Reg(i_* F)\leq Qreg(F)+1.$$
   \begin{proof}
   We have to prove that:\\
   $F$ $m$-regular $\Rightarrow$ $F$ $m$-Qregular and\\
   $F$ $(m-1)$-Qregular $\Rightarrow$ $F$ $m$-regular.\\
   For any integer $t$ and for any $i>0$ we have:
   $$H^i(\mathbb P^{n+1},i_* F(t))=H^i(\Q_n,F(t)),$$
   so $F$ is $m$-regular if and only if $H^i(\Q_n,F(m-i))=0$ for $i=1\dots n$.\\
   Let $F$ be $m$-regular we only need to prove that $H^n(\Q_n,F(m)\otimes \Sigma_*(-n))=0$.\\
   From the exact sequence $$ 0\rightarrow F(m)\otimes \Sigma_*(-n-1)\rightarrow F(m-n)^{2^{([n/2]+1)}} \rightarrow F(m)\otimes\Sigma_*(-n)\rightarrow 0,$$ we see that $H^n(\Q_n,F(m-n))=0\Rightarrow H^n(F(m)\otimes\Sigma_*(-n))=0$.\\
   Let $F$ be $(m-1)$-Qregular, by (\ref{p1}) $F$ is $m$-Qregular, so we only need to prove that\\ $H^n(\Q_n,F(m-n))=0$.\\
   From the exact sequence $$ 0\rightarrow F(m-1)\otimes \Sigma_*(-n)\rightarrow F(m-n)^{2^{([n/2]+1)}} \rightarrow F(m)\otimes\Sigma_*(-n)\rightarrow 0,$$ we see that $H^n(F(m-1)\otimes\Sigma_*(-n))=H^n(F(m)\otimes\Sigma_*(-n))=0\Rightarrow H^n(\Q_n,F(m-n))=0$.\\
  \end{proof}\end{proposition}
  \begin{remark} The above Proposition is optimal because for instance $Qreg (\sO)=Qreg(\Sigma_*)=0$ by Remark \ref{r1} but $Reg (i_*\sO)=1$ and $Reg(i_*\Sigma_*)=0$ on $\Q_n$ ($n>2$).\\
  In fact $$H^n(\mathbb P^{n+1},i_* \sO(t-n))=H^n(\Q_n,\sO(t-n))=0$$ if and only if $t\geq 1$ and\\
  $$H^n(\mathbb P^{n+1},i_* \Sigma_*(t-n))=H^n(\Q_n,\Sigma_*(t-n))=0$$ if and only if $t\geq 0$.
  \end{remark}
  We are ready to classify the coherent sheaves with Qregularity $-\infty$:\\
  \begin{theorem}
Let $F$ be a coherent sheaf on $\Q_n$ ($n$ even).\\ The following  condition are equivalent:\\
\begin{enumerate}
\item $Qreg (F)=-\infty$.\\
\item $Reg (F)=-\infty$.\\
\item $\mbox{Supp}(F)$ is finite.\\
\end{enumerate}
Let $F$ be a coherent sheaf on $\Q_n$ ($n$ odd). Let us consider the geometric collection on $\Q_n$: $$\sigma= (\sO,\dots , \sO(n-1), \Sigma(-n-1)).$$ The following  condition are equivalent:\\
\begin{enumerate}
\item $Qreg (F)=-\infty$.\\
\item $Reg (F)=-\infty$.\\
\item $Reg_{\sigma} (F)=-\infty$.\\
\item $\mbox{Supp}(F)$ is finite.\\
\end{enumerate}

\begin{proof}
Let $n$ be an even integer and let $F$ be a coherent sheaf on $\Q_n$.\\
By (\ref{p2})  $Qreg (F)=-\infty$ if and only if $Reg (F)=-\infty$.\\
By \cite{bm} Theorem $1.$  $Reg (F)=-\infty$ if and only if $\mbox{Supp}(F)$ is finite.\\

Let $n$ be an odd integer and let $F$ be a coherent sheaf on $\Q_n$.\\
By (\ref{p2})  $Qreg (F)=-\infty$ if and only if $Reg (F)=-\infty$.\\
By \cite{cm1} Theorem $4.3.$  $Reg (F)=-\infty$ if and only if $Reg_{\sigma} (F)=-\infty$.\\
By \cite{bm} Theorem $1.$  $Reg_{\sigma} (F)=-\infty$ if and only if $\mbox{Supp}(F)$ is finite.\\

\end{proof}
\end{theorem}

\section{Evans-Griffiths criterion  on quadrics}
 We use our notion of regularity in order to  proving our main result:

{\emph {Proof of Theorem \ref{t1}.}} Let assume that $E$ is Qregular but $E(-1)$ not.\\Here we use the definition of Qregularity as in Remark \ref{d-e}.\\
Since $E$ is Qregular, it is globally generated and $E(1)$ is ample. So, by Le Potier vanishing theorem, we have that $H^i(E^{\vee}(-l))=0$ for every $l>0$ and $i=1, \dots, n-r$.\\
 So by Serre duality $H^i(E(-n+l))=0$ for every $l>0$ and $i=r, \dots, n-1$.\\
  In particular $H^i(E(-1-i))=0$ for $i=r, \dots, n-2$ and by hypothesis   $H^i(E(-1-i))=0$ for $i=1, \dots, n-r$ and  $H^{n-1}(E(-1-n+1) =0$.\\ We can conclude that $E(-1)$ is not Qregular if and only if   $H^{n-1}(E(-1)\otimes \Sigma_*(-n+1))\not=0$ or $H^n(E(-1-n+1))\not=0$.\\Let assume first that $H^n(E(-1-n+1))\not=0$, this means by Serre duality that $H^0(E^{\vee})\not=0$. We have a non zero map $$  f:E\rightarrow \sO$$ Now, since $E$ is globally generated, we have the exact sequence
 $$ \sO\rightarrow \sO\otimes H^0(E) \rightarrow E\rightarrow \sO.$$
 The composition of the maps is not zero so must be the identity and we can conclude that $\sO$ is a direct summand of $E$.\\
      Let assume now that $H^{n-1}(E(-1)\otimes \Sigma_*(-n+1))\not=0$ and $H^n(E(-1-n+1))=0$.\\
      Let see first the even case: let $n=2m$ and $H^{n-1}(E(-1)\otimes \Sigma_1(-n+1))\not=0$.\\
       We consider the following exact sequences:  $$ 0\rightarrow E(k)\otimes \Sigma_2(-1)\rightarrow E(k)^{2^{([n/2]+1)}} \rightarrow E(k)\otimes\Sigma_1\rightarrow 0.$$Since $ H^n(E(-1-n+1))= H^{n-1}(E(-1-n+1))=0$, we see that  $$H^{n-1}(E(-1)\otimes \Sigma_1(-n+1))\cong H^{n}(E(-1)\otimes \Sigma_2(-n))$$ so, by Serre duality  $H^0(E^{\vee}(1)\otimes \Sigma^{\vee}_2)\not=0$ and there exists a non zero map$$  f: E(-1)\rightarrow \Sigma^{\vee}_2.$$
      On the other hand, since for any $j=1, \dots, n-1$  $H^j(E(-1-j))=0$, the following maps$$ H^0(E\otimes \Sigma_2(-1))  \rightarrow    H^{1}(E\otimes \Sigma_1(-2)) \rightarrow  \dots  \rightarrow  H^{n-2}(E\otimes \Sigma_2(-n+1)) \rightarrow  H^{n-1}(E\otimes \Sigma_1(-n))$$ are all surjective.\\ In particular we have that $H^0(E(-1)\otimes\Sigma_2) \not=0$ and there exists a non zero map$$  g: \Sigma^{\vee}_2 \rightarrow  E(-1) .$$ Let us consider the following commutative diagram:$$\begin{array}{ccc} H^{n-1}(E(-1)\otimes \Sigma_1(-n+1)) \otimes H^{1}(E^{\vee}(1)\otimes \Sigma^{\vee}_1(-1))& \xrightarrow{\sigma}& H^{n}(\Sigma_1^{\vee}(-1)\otimes \Sigma_1(-n+1))\cong\mathbb C\\ \downarrow & &\downarrow \\H^0(E(-1)\otimes \Sigma_2)  \otimes H^{1}(E^{\vee}(1)\otimes \Sigma^{\vee}_1(-1))& \xrightarrow{\mu}& H^{1}(\Sigma_1^{\vee}(-1)\otimes \Sigma_2)\cong\mathbb C\\\downarrow & &\downarrow \\
      H^0(E(-1)\otimes \Sigma_2)  \otimes H^{0}(E^{\vee}(1)\otimes \Sigma^{\vee}_2)& \xrightarrow{\tau}& H^{0}(\Sigma_2^{\vee}\otimes \Sigma_2)\cong\mathbb C\\
    \uparrow{\cong}&&\uparrow{\cong}\\
     {Hom}(\Sigma_2^\vee,E(-1))\otimes{Hom}(E(-1),\Sigma_2^\vee)&\xrightarrow{\gamma}&{Hom}(\Sigma_2^\vee,\Sigma_2^\vee)    \end{array}$$The map $\sigma$ comes from Serre duality and it is not zero, the right vertical map are isomorphisms and the left vertical map are surjective so also the map $\tau$ is not zero.\\  This means that the composition of the maps $f$ and $g$ is not zero so must be the identity and we can conclude that $\Sigma^{\vee}_2$ is a direct summand of $E(-1)$.\\
     By \cite{o1} Theorem $2.8.$ we have
     $$\Sigma^{\vee}_2\cong
     \left\{\begin{array}{cc}
     \Sigma_2(1)& \textrm{if $m\equiv 0$ (mod $4$)}\\
      \Sigma_1(1)& \textrm{if $m\equiv 2$ (mod $4$)}\end{array}\right.$$
      In the same way we can prove that, if $H^{n-1}(E(-1)\otimes \Sigma_2(-n+1))\not=0$,  $\Sigma^{\vee}_1$ is a direct summand of $E(-1)$; or in the odd case that $\Sigma^{\vee}$ is a direct summand of $E(-1)$.\\
      By iterating these arguments we have that $E$ is a direct sum of line bundles and spinor bundles with some twist.\qquad
      As a Corollary we get the following splitting criterion:
       \begin{corollary} Let $E$ be a rank $r$ vector bundle on $\Q_n$ such that $H^i_*(E) =0$ for any $i=1, \dots, r-1$,  $H^{n-1}_*(E) =0$  and $H^{n-1}_*(E\otimes \Sigma_*)=0$, then $E$ is a direct sum of line bundles.
       Let $E$ be a rank $r$ vector bundle on $\Q_n$. Then the following conditions are equivalents:
\begin{enumerate}
\item $H^i_*(E) =0$ for any $i=1, \dots, r-1$, $H^{n-1}_*(E) =0$ and $H^{n-1}_*(E\otimes \Sigma_*)=0$,
\item $E$ is a direct sum of line bundles with some twist.
\end{enumerate} \end{corollary}
\begin{corollary}[Kn\"orrer] Let $E$ be a rank $r$ vector bundle on $\Q_n$ such that $H^i_*(E) =0$ for any $i=1,\dots, n-1$, then $E$ is a direct sum of line bundles and spinor bundles with some twist.  \end{corollary}
      \begin{remark} The hypothesis $H^{n-1}_*(E) =0$ does not appear in the Evans-Griffiths criterion on $\mathbb P^{n}$. On $\Q_n$ it is necessary.\\
       In fact we can find many indecomposable bundles with  $H^{1}_*(E)=\dots =H^{n-2}_*(E)=0$ but
      $H^{n-1}_*(E) \not=0$.\\
      For instance on $\Q_4$ there is the rank $3$ bundle $P_4$ arising from the following exact sequence (see \cite{h} or \cite{ml}):
      $$0 \to \sO \rightarrow
\Sigma_1\oplus\Sigma_2 \rightarrow P_4 \to 0.$$
On $\Q_5$ there is the rank $3$ bundle $P_5$ arising from the following exact sequence (see \cite{ml}):
      $$0 \to \sO \rightarrow
\Sigma \rightarrow P_5 \to 0.$$
      \end{remark}
      For rank $2$ bundles, since $E^\vee\cong E(c_1)$, the hypothesis $H^{n-1}_*(E) =0$ is not necessary.
We can also prove the following result:\\
\begin{proposition}\label{p4}Let $E$ be a rank $2$ bundle on $\Q_n$ with $Qreg (E)= 0$ and $H^1(E(-2))=H^1(E(c_1))=0$.\\
Then $E$ is a direct sum of line bundles and spinor bundles with some twist.\\
If $n>4$, $E\cong\sO\oplus\sO(c_1)$.
\end{proposition}
     \begin{proof}
Since $E$ is Qregular, it is globally generated and $E(1)$ is ample. So, by Le Potier vanishing theorem, we have that $H^i(E^{\vee}(-l))=0$ for every $l>0$ and $i=1, \dots, n-2$.\\
 So by Serre duality $H^i(E(-n+l))=0$ for every $l>0$ and $i=2, \dots, n-1$.\\
 In particular $H^i(E(-1-i))=0$ for $i=r, \dots, n-2$ and by hypothesis   $H^i(E(-1-i))=0$ for $i=1, \dots, n-r$ and  $H^{n-1}(E(-1-n+1))\cong H^1(E(c_1))=0$.\\ We can conclude that $E(-1)$ is not Qregular if and only if   $H^{n-1}(E(-1)\otimes \Sigma_*(-n+1))\not=0$ or $H^n(E(-1-n+1))\not=0$.\\
Now arguing as in the above theorem we can conclude that $E$ contains $\sO$ as a direct summand if $n>4$.\\
If $n<5$ $E$, since the rank of the spinor bundles is smaller than $3$, can also contains $\Sigma_*$ as a direct summand.
\end{proof}

\end{document}